\documentclass{amsart}
\usepackage{enumerate}
\usepackage{amsthm,amscd,amssymb,verbatim,epsf,amsmath,amsfonts,mathrsfs,graphicx}
\usepackage[linktocpage,colorlinks=true,linkcolor=blue,citecolor=blue]{hyperref}
\usepackage{lineno}
\newtheorem{thm}{Theorem}

\newtheorem{prop}{Proposition}

\DeclareMathOperator{\dist}{dist}

\title[Curvature estimates]{Curvature estimates of spacelike surfaces in de Sitter space.}

\author{Daniel Ballesteros-Ch\'avez}
\address{Daniel Ballesteros-Ch\'avez\\ 
Dep. Applied Mathematics \\
Silesian University of Technology\\
           Gliwice, 44-100\\
          Poland.}
\email{daniel.ballesteros-chavez@polsl.pl}

\thanks{The author was supported by CONACYT-Doctoral scholarship no.~411485}
\date{}

\begin{document}
\maketitle


\begin{abstract}
 Local estimates of the maximal curvatures of admissible spacelike hypersurfaces in de Sitter space
 for k-symmetric curvature functions are obtained. They depend on interior and boundary data.
\end{abstract}

\section{Introduction}
In this work we will consider solutions to fully nonlinear PDEs  of the form
\begin{equation}
  \label{eq:1}
  F(A) = f(\lambda_1,\ldots,\lambda_n) = \psi, \mbox{ in } \Omega\subset\mathbb{S}^n,
\end{equation}
where $A$ is the second fundamental form of a spacelike hypersurface in de Sitter space $S^{n+1}_1$.
Furthermore $f$ is a symmetric function of the eigenvalues of $A$, and $\psi$ is a function of the
position vector and the tilt of the hypersurface to be defined below.
We will assume that the hypersurface is the graph of a function over an open set of the sphere.
More precisely, let $\Omega \subset \mathbb{S}^n$ be an open set and $u : \Omega \to I$ a
smooth function, where $I=[R_1,R_2]$ is the real interval $0<R_1<R_2$, such that the graph
\begin{equation}
  \label{eq:2}
  \Sigma = \mbox{graph}(u) = \left\{ Y= ( u( \xi ),\xi ) \, | \,   \xi\in\Omega \right\} \subset S^{n+1}_1
\end{equation}
is a spacelike hypersurface in de Sitter space $S^{n+1}_1$.

For $1\leq k \leq n$ and $\lambda = (\lambda_1,\ldots,\lambda_n)\in\mathbb{R}^n$, let  $S_{k}(\lambda) =  \Sigma_{1\leq i_1<\cdots<i_k\leq n}\lambda_{i_1}\cdots\lambda_{i_k} $, and
define the normalised symmetric polynomial $H_k(\lambda) = {n\choose k}^{-1}S_{k}$. In this paper we
consider the case when $f$ is the homogeneous function of degree one given by
\begin{equation}
  \label{eq:3}
  f(\lambda) = H_{k}^{1/k}(\lambda) ,
\end{equation}
defined in an open convex cone $\Gamma$ which is symmetric, with vertex at the origin
and contains the positive cone
$\Gamma^{+}=\{\lambda\in\mathbb{R}^n\,|\, \lambda_i>0, \forall i = 1,2,\dots,n\}$.

Since $f \in C^2(\Gamma)\cap C^{0}(\bar{\Gamma})$,  $f_{\lambda_i} >0 $ for all $i = 1,2,\ldots,n$, and $f(\lambda)$ is
concave in $\Gamma$, it follows that $F$ is elliptic and concave.
A solution
$u$ will be called \emph{admissible} if the principal curvatures $\lambda = (\lambda_1,\ldots,\lambda_n)$ of the spacelike hypersurface $\Sigma$
given by \eqref{eq:2} belong to the connected component of $\Gamma_k$ containing
$\Gamma^+$, where $\Gamma_{k}:= \{\lambda\in\mathbb{R}^{n} \,|\, H_{k}(\lambda)>0\}$.

The existence of solutions of such equations has been studied 
in \cite{CNS3} by L. Caffarelli, L. Nirenberg and J. Spruck.
In \cite{CNS4}, they proved the existence of starshaped hypersurfaces in Euclidean space
with prescribed $k$-symmetric curvature using the a priori $C^{2,\alpha}$ estimate needed to carry out
the continuity method. By the Evans-Krylov theorem it is sufficient
obtain the apriori $C^{0}, C^{1}$ and $C^2$ estimates for admissible solutions, where
the last one follows from an estimate of the maximal principal curvature of the hypersurface.

For various ambient Riemannian manifolds, curvature estimates for
starshaped hypersurfaces with given $k$-symmetric curvature have also
been proved. Namely for hypersurfaces in the sphere, the lower order and
the curvature estimate are given in \cite{Oliker1} by M. Barbosa,
L. Herbert and V. Oliker. These were used for the existence result by
Y. Li and V. Oliker in \cite{YYLi:elli}. The curvature estimate and the
existence result for hypersurfaces in the hyperbolic space was
proved by Q. Jin and Y. Li in \cite{YYLi:hyp} using similar
arguments of W. Sheng, J. Urbas and X. Wang in \cite{sheng2004}. The
lower order estimates for this case are also contained in
\cite{Oliker1} and used to complete the existence result. For
spacelike hypersurfaces in
Minkowski space and Lorentz manifolds
various results have been proved by  R. Bartnik and L. Simons
\cite{bartnik1982}, C. Gerhardt \cite{CG:Lorentz, Gerhardt1983, Gerhardt1997},
Y. Huang \cite{Huang:01} and the references provided in them.

We obtain similar curvature estimates as in \cite{Huang:01}
in de Sitter space. As in \cite{Huang:01} we impose a growth assumption on the right hand side in terms of the tilt
$\tau$  (see \eqref{eq:39}). We introduce in Section \ref{sec:geom-form-hypers}
the geometric formulae of hypersurfaces in Lorentzian Manifolds,
and provide explicit expressions for hypersurfaces in de Sitter  space. In Section \ref{sec:proof-theorem-1}
we prove the following
\begin{thm}
  \label{thm:01}
  Let $\Omega\subset\mathbb{S}^n$ be a domain in the round sphere, and let $u\in C^{4}(\Omega) \cap C^2(\bar{\Omega})$ an admissible solution of
  the boundary value problem
  \begin{equation*}
    \label{eq:t01} 
    \left\{
      \begin{array}{rcll}
        F(A) =  H_{k}^{\frac{1}{k}}(\lambda(A)) &=& \psi(Y,\tau) & \mbox{in}\quad \Omega\\
        u  &=& \varphi & \mbox{on}\quad \partial\Omega
      \end{array}
    \right.,
  \end{equation*}
  where $A$ is the second fundamental form of a spacelike surface $\Sigma$ in de Sitter space given by
  (\ref{eq:25}), $\psi\in C^{\infty}(\bar{\Omega})$, $\psi>0$ and convex in $\tau$. Assume additionally that 
  \begin{equation*}
    \label{eq:45}
    \psi_{\tau}(X,\tau)\tau - \psi(X,\tau) \geq 0,   
  \end{equation*}
  for all $X\in S^{n+1}_1$ and $\tau \in [1,\infty)$. Then  
  \begin{equation*}
    \label{eq:46}
    \sup_{\Omega}|A| \leq C,
  \end{equation*}
  where $C$ depends on $n$, $\|\varphi\|_{C^1(\bar{\Omega})}$, $\|\psi\|_{C^2(I\times\Omega\times [1,\infty))}$ and $\sup_{\partial\Omega}|A|$. 
\end{thm}

And finally in Section \ref{sec:proof-theorem-2} we
give an interior estimate when the growth condition
is strict and the boundary data is spacelike and
affine. 

\begin{thm}
  \label{thm:02}
      Let $\Omega\subset\mathbb{S}^n$ be a domain in the round sphere, and let $u\in C^{4}(\Omega) \cap C^2(\bar{\Omega})$ an admissible solution of
                 the boundary value problem
                 \begin{equation*}
                   \label{eq:51}
                   \left\{
                     \begin{array}{rcll}
                      F(A) =  H_{k}^{\frac{1}{k}}(\lambda(A)) &=& \psi(Y,\tau) & \mbox{in}\quad \Omega\\
                       u  &=& \varphi & \mbox{on}\quad \partial\Omega
                     \end{array}
                     \right.,
                   \end{equation*}
                   where $A$ is the second fundamental form of a spacelike surface $\Sigma$ in de Sitter space given by
                   (\ref{eq:25}), $\psi\in C^{\infty}(\bar{\Omega})$, $\psi>0$ and convex in $\tau$. Assume also that 
                   \begin{equation*}
                     \label{eq:52}
                     \psi_{\tau}(X,\tau)\tau - \psi(X,\tau) > 0,   
                   \end{equation*}
                   for all $X\in S^{n+1}_1$ and $\tau \in [1,\infty)$,  
                   and that the domain $\Omega$ is
                   $C^2$, uniformly convex. If the boundary value
                   $\varphi$ is spacelike and affine, namely $\varphi$ is the restriction of an affine function
                   on ambient Minkowski space of $n+2$ dimension. Then for any
                   $\Omega'\subset\subset\Omega$, there is a constant
                   $C$ depending only on $n, \Omega$,
                   $\dist(\Omega',\partial\Omega)$, $\|\varphi\|_{C^1(\bar{\Omega})}$ and $\|\psi\|_{C^2(I\times \Omega\times [1,\infty))}$, such that 
                   \begin{equation*}
                     \label{eq:53}
                     \sup_{\Omega'}|A| \leq C.
                   \end{equation*}
                 \end{thm}

\section{Geometric formulae for hypersurfaces in de Sitter space}
\label{sec:geom-form-hypers}

We will recall some geometric formulae for hypersurfaces in Lorentzian manifolds and at the
end we will apply them to the case of spacelike hypersurfaces in de Sitter space.

Let $\{\partial_1,...,\partial_n,N\}$ be a coordinate frame for a Lorentzian
manifold $(\bar{M},\bar{g})$ and $M$ a Lorentzian  (not necessarily spacelike) hypersurface with induced metric $g$ such that
$\{\partial_i\}$ span $TM$, and let $N$ be the unit normal field to $M$
and put $\epsilon = \bar{g}(N,N)$. When the induced metric is positive
definite, then we say that $M$ is a spacelike hypersurface. The metric $g$
can be represented by the matrix
$g_{ij} = g(\partial_i,\partial_j)$ with inverse denoted by $g^{ij}$.

The \textit{Gauss formula} for $X,Y\in T\Sigma$ reads
\begin{equation*}
  \label{eq:4}
  D_{X}Y = \nabla_{X}Y +\epsilon\, h(X,Y)N,
\end{equation*}
here $D$ is the connection on $\bar{M}$, $\nabla$ is the induced connection on $M$ and the \textit{second fundamental form} $h$ is 
the normal projection of $D$. In a coordinate basis we write
\begin{equation*}
  \label{eq:5}
h_{ij}= h(\partial_i,\partial_j).
\end{equation*}
The \textit{shape operator} is obtained by raising an index with the inverse of the metric
\begin{equation*}
  \label{eq:6}
  h^i_j = g^{ik}h_{kj}.
\end{equation*}
The \textit{principal curvatures} of the hypersurface $\Sigma$ are the
eigenvalues of the symmetric matrix $(h^i_j)$. The tangential projection of the
covariant derivative of the normal vector field $N$ on $\Sigma$, $\nabla_j N = (D_{\partial_j}N)^{\top}$,
is related to the second fundamental form by the \textit{Weingarten equation}
      \begin{equation}
        \label{eq:7}
        \nabla_{j} N = -h^i_j\partial_i = -g^{ik}h_{kj}\partial_i.
      \end{equation}
      The \emph{curvature tensor} is defined for $X,Y,Z \in T\Sigma$ as
      
      \begin{equation*}
        \label{eq:8}
        R(X,Y)Z = \nabla_{Y}\nabla_{X}Z - \nabla_{X}\nabla_{Y}Z + \nabla_{[X,Y]}Z.
      \end{equation*}
      The \emph{Christoffel symbols} are given by 
        \begin{equation}
          \label{eq:9}
          \Gamma_{ij}^k = \frac{1}{2}g^{kl}\left(\partial_ig_{jl} + \partial_{j}g_{il} - \partial_lg_{ij}\right),
        \end{equation}
      and the curvature tensor in terms of Christoffel symbols is
      \begin{equation*}
        \label{eq:10}
        R_{ijk} = R^{m}_{ijk}\partial_m =
        \left(\partial_j\Gamma_{ik}^m -
          \partial_i\Gamma_{jk}^m + \Gamma^m_{js}\Gamma^s_{ik} - \Gamma^m_{is}\Gamma^s_{jk}\right)\partial_m.
      \end{equation*}
      Contracting with the metric
      \begin{equation*}
        \label{eq:11}
        R_{ijkl} = g\left(R(\partial_i,\partial_j)\partial_k,\partial_l\right) = g_{lm}R^{m}_{ijk}.
        \end{equation*}
      We can also write the  curvature tensor of the ambient manifold in terms of
      the curvature of the surface and the second fundamental form
      \begin{equation*}
        \label{eq:12}
        \begin{split}
          \bar{R}_{ijk} &= R^m_{ijk}\partial_m\\
          &=D_j(D_i\partial_k) - D_i(D_j\partial_k)\\
          &=(\nabla_j+D^{\perp}_j)(\nabla_i\partial_k + \epsilon h_{ik}N) - (\nabla_i+D^{\perp}_i)(\nabla_j\partial_k + \epsilon h_{jk}N) \\
          &=R_{ijk}+ \epsilon h_{ik}\nabla_jN - \epsilon h_{jk}\nabla_iN + \epsilon D^{\perp}_{j}(hN)_{ik} - \epsilon D^{\perp}_{i}(hN)_{jk},
          \end{split}
      \end{equation*}
      where $D^{\perp}_i(hN)_{jk}= D^{\perp}_{i}(h_{jk}N) -\Gamma_{ik}^rh_{rj}N -\Gamma_{ij}^rh_{rk}N$.
      
 From the last identity, when the ambient manifold is flat, we obtain the \textit{Codazzi equation}  given by the identity 
\begin{equation}
  \label{eq:13}
  \nabla_i h_{jk} = \nabla_j h_{ik}.
\end{equation}
Note that the first and second covariant derivatives of the second fundamental form are given by
\begin{equation*}
  \begin{split}
        \label{eq:14}
        \nabla_l h_{ij} &= \partial_l h_{ij} - \Gamma_{li}^rh_{rj} - \Gamma_{lj}^rh_{ir},\\
        \nabla_{k}\nabla_l h_{ij} &= \partial_k(\nabla_l h_{ij}) - \Gamma_{kl}^r\nabla_rh_{ij} - \Gamma_{ki}^r\nabla_{l}h_{rj}- \Gamma_{kj}^r\nabla_{l}h_{ir}.
      \end{split}
    \end{equation*}

\textit{The Gauss Equation} expressed in orthonormal coordinates, is given by
\begin{equation}
  \label{eq:15}
  \bar{R}_{ijkl} = R_{ijkl} - \epsilon \left(h_{ik}h_{jl} - h_{il}h_{jk} \right).
\end{equation}

When $M$ is a hypersurface of a flat manifold $\bar{R}_{lkij} = 0$,
the last equation simplifies to the identity
\begin{equation*}
  \label{eq:16}
  R_{ijkl} =\epsilon\left( h_{ik}\,h_{jl}- h_{jk}\,h_{il} \right).
\end{equation*}

Note that $A$ is a bilinear symmetric tensor, and the following \textit{Ricci identity} holds
      \begin{equation}
        \label{eq:17}
{\nabla_{k}}\nabla_{l}A_{ij} - {\nabla_{l}}\nabla_{k}A_{ij} =  R_{kljr}A_{ir} + R_{klir}A_{rj}.
      \end{equation}

Let $\mathbb{R}^{n+2}_{1}=(\mathbb{R}^{n+2},\bar{g})$ be the Minkowski space with metric $\bar{g} = -dx_1^2 + dx_2^2 + \cdots + dx_{n+2}^2$ and
covariant derivative $\bar{D}$.
Then \emph{de Sitter} space is defined as
${S}^{n+1}_{1} = \left\{ x\in\mathbb{R}^{n+2}_{1} : \bar{g}( x,x) = 1\right\}$ with the
induced Lorentzian metric which we will denote by $g$, and covariant derivative $D$. Moreover,
any point in ${S}^{n+1}_1$ can be written as
$(r,\xi)\in\mathbb{R}^{+}\times \mathbb{S}^n$, 
with the induced metric
\begin{equation*}
  \label{eq:18}
  g = -dr^2 + \cosh^2(r)\sigma,
\end{equation*}
where $\sigma$ is the round metric on $\mathbb{S}^n$, and later we will use
$\tilde{\nabla}$ to denote the covariant derivative for the metric $\sigma$.
The vector field $\partial_{r}$ will be written separately from any
other index notation $\partial_{\alpha},\partial_j,...$, etc., the latter indices taking values form $1$ to $n$.

Let $u:\mathbb{S}^{n}\to  [0,\infty)$ be a smooth function and consider a
spacelike hypersurface in ${S}^{n+1}_1$ given by the graph
$\Sigma=\{(u(\xi),\xi)\}$.
The tangent space of the hypersurface at a point $Y\in\Sigma$
is spanned by the tangent vectors
$Y_j = u_{j}\partial_{r} + \partial_{j}$, the covariant derivative
 $\nabla$ corresponding to the induced metric 
on $\Sigma$  which is given by 
\begin{equation*}
  \label{eq:19}
  G_{ij} = -u_{i} u_{j} + \cosh^2(u)\sigma_{ij} .
\end{equation*}
Since the metric is positive definite, its inverse can be computed
  \begin{equation*}
    \label{eq:20}
    \begin{split}
      G^{ij}
    &=\cosh^{-2}(u) \sigma^{ij} +
    \frac{ \sigma^{i \gamma}u_{\gamma} \sigma^{j\eta}u_{\eta}}{\cosh^4(u) - \cosh^{2}(u)|\tilde{\nabla} u|^2},
    \end{split}
  \end{equation*}
where $\tilde{\nabla}u = \sigma^{ij}u_j\partial_i$ and $|\tilde{\nabla}u| := \sigma^{ij}u_iu_j$.  Note that for this to be well defined we need
  to have $|\tilde{\nabla} u|^2 \neq \cosh^2(u)$, and this is the case when the surface is spacelike.  A unit normal vector
  to $\Sigma$ at the point $Y$ can be obtained by solving the equation $g(Y_{\alpha},\hat{n})=0$, and
  then we get
  
\begin{equation*}
  \label{eq:21}
  \hat{n} =- \frac{ \cosh^2(u)\partial_{r}+ \tilde{\nabla} u }{
    \sqrt{\epsilon\left(-\cosh^4(u) + \cosh^2(u)|\tilde{\nabla} u|^2\right)}},
\end{equation*}
and moreover, since $\Sigma$ is spacelike, then the following inequality must hold
  \begin{equation}
    \label{eq:22}
    |\tilde{\nabla} u | \leq \cosh(u),
  \end{equation}
because the unit vector $\hat{n}$ normal to $\Sigma$ is time-like, that is
$g(\hat{n},\hat{n}) = -1$.

The second fundamental form is the projection of the second derivatives of the
parameterisation $D_{Y_{\alpha}}Y_{\beta}$ on the normal direction. Notice that
from \eqref{eq:9}, and writing $\tilde{\Gamma}$ for the Christoffel symbols
of the metric $\sigma$, we have
  \begin{equation*}
    \label{eq:23}
    D_{\partial_r}\partial_r = 0; \quad
    D_{\partial_r}{\partial_j} = \tanh(r)\partial_j;\quad
    D_{\partial_i}{\partial_j}= \cosh(r)\sinh(r)\sigma_{ij}\partial_{r}+ \tilde\Gamma_{ij}^k \partial_k ,
  \end{equation*}
and using these identities we compute
  \begin{equation*}
    \label{eq:24}
    \begin{split}
    D_{Y_i}Y_{j}& = D_{u_i\partial_{r}+\partial_i}\left(u_j\partial_{r}+\partial_j\right)\\
    &=
    u_ju_jD_{\partial_{r}}\partial_{r} +
    u_iD_{\partial_{r}}\partial_{j} +
    u_{ij} \partial_{r} +
    u_{j}D_{\partial_{i}}\partial_{r} +
    D_{\partial_{i}}\partial_{j}.
    \end{split}
  \end{equation*}
Let $W^2 = \cosh^4(u) - \cosh^2(u)|\tilde{\nabla} u|^2$, then  $ A_{ij}=  g(D_{Y_{i}}Y_{j}, \hat{n})$ is given explicitly by
 \begin{equation}
  \label{eq:25}
  A_{ij} = \frac{\cosh^2(u)}{W}
  \left(\tilde{\nabla}^2_{ij}u - 2\frac{\sinh(u)}{\cosh(u)}u_{i}u_{j} + \sinh(u)\cosh(u)\sigma_{ij}\right).
\end{equation}

Recalling that the Minkowski space is a flat Lorentzian
manifold, and letting $h$ denote the second fundamental form
of de Sitter space $S^{n+1}_1$, when we apply the Gauss equation \eqref{eq:15}
to the surface as a submanifold of
codimension two $\Sigma\subset {S}^{n+1}_1\subset\mathbb{R}^{n+1,1}$, we have
\begin{equation}
  \label{eq:26}
  \begin{split}
    0= \bar{\bar{R}}_{ijkl} &=  \bar{R}_{ijkl} -\epsilon_1 (h_{ik}h_{jl} - h_{il}h_{jk})\\
    &= R_{ijkl} - \epsilon_2 (A_{ik}A_{jl} - A_{il}A_{jk}) - \epsilon_1(h_{ik}h_{jl} - h_{il}h_{jk}).
  \end{split}
\end{equation}

The Gauss formula applied twice reads
  \begin{equation}
  \label{eq:Gausscod2}
  D_{Y_i} Y_j = \nabla_{Y_{i}}Y_{j} - A_{ij}\hat{n} - \langle Y_i,Y_j \rangle Y.
\end{equation}

For any function $f :S^{n+1}_1 \times \mathbb{R} \to \mathbb{R}$, the partial derivative on $S^{n+1}_{1}$ and $\Sigma$ are
defined respectively as
\begin{equation}
D^{x}f = \overline{g}^{\alpha\beta} \frac{\partial f}{\partial x_{\alpha}} \partial_{\beta}, \mbox{and} \quad \nabla^{x}f = (D^xf)^{\top}.
\end{equation}

Finally let us remark that at a given point of $\Sigma$ we can use coordinates such that
the second fundamental form $\{A_{ij}\}$ is diagonal, thus $\lambda_i = A_{ii}$ at the point,
and through the paper we assume $\lambda_1\geq\cdots\geq\lambda_n$, and we may also assume that $\lambda_1\geq 1$. The fact that $A$ is diagonal at a point also implies that
$F^{ij}:=\frac{\partial F}{\partial A_{ij}}$ is also 
diagonal and we can also write $F^{ii}= f_{i}$.

\section{Proof of Theorem 1}
\label{sec:proof-theorem-1}

We are now going to prove that if $u$ is an admissible solution of
\eqref{eq:t01} then the curvature of the hypersurface is bounded, then
the $C^2$ 
estimate of the solution will be a consequence of the equation of the
second fundamental form (\ref{eq:25}) and lower order estimates. We
will need 
the commutator formula for second order derivatives of the second
fundamental form, given by Ricci's identity~(\ref{eq:17}), together
with the Gauss equation of the surface as a codimension $2$
spacelike submanifold of the Minkowski space. With this in account and together with
equation (\ref{eq:26}) we obtain the following
      \begin{equation}
        \label{eq:27}
           R_{ijkl} = - (A_{ik}A_{jl} - A_{il}A_{jk}) + (h_{ik}h_{jl} - h_{il}h_{jk}),
      \end{equation}
where we are using $A_{ij}$ for the second fundamental form of the spacelike 
      hypersurface in de Sitter space, and $h_{ij}$ denotes the second fundamental form
      of de Sitter space in flat Minkowski space.
      Substituting in equation (\ref{eq:17}) we get
         \begin{equation*}
           \label{eq:28}
           \begin{split}
          \nabla_{k}\nabla_{l} A_{ij} &= \nabla_{l}\nabla_{k} A_{ij} + \sum_{r}R_{kljr}A_{ir}+ \sum_{r} R_{klir}A_{rj} \\
             &=\nabla_{l}\nabla_{k} A_{ij} + \sum_r\left\{- (A_{kj}A_{lr} - A_{kr}A_{lj}) + (h_{kj}h_{lr} - h_{kr}h_{lj})\right\}A_{ir}\\
             &\qquad+\sum_r\left\{- (A_{ki}A_{lr} - A_{kr}A_{li}) + (h_{ki}h_{lr} - h_{kr}h_{li})\right\}A_{rj}.\\
             \end{split}
      \end{equation*}

Moreover, notice that by the Codazzi equation, the Ricci identity (\ref{eq:17}) and summing over $r$ we get
           \begin{equation*}
             \label{eq:29}
             \begin{split}
               \nabla_{i}\nabla_{j}A_{kk} =& \nabla_{i}\nabla_{k}A_{kj} \\
               =&  \nabla_{k}\nabla_{i} A_{kj} + R_{ikkr}A_{rj} + R_{ikjr}A_{kr}\\
               =& \nabla_{k}\nabla_{k} A_{ij} + R_{ikkr}A_{rj} + R_{ikjr}A_{kr}.
               \end{split}
           \end{equation*}
           Using coordinates such that  $A$ is diagonal, form equation (\ref{eq:27}) we obtain
           \begin{multline}
             \label{eq:30}
             \nabla_{j}\nabla_{j} A_{kk} = \nabla_{k}\nabla_{k}A_{jj} + A_{kk}A_{jj}^2 + h_{jk}h_{jk}A_{jj} - h_{kk}h_{jj}A_{jj}\\
             - A_{jj}A_{kk}^2 + h_{jj}h_{kk}A_{kk} - h_{jk}h_{jk}A_{kk}.
           \end{multline}

      The first covariant derivative of (\ref{eq:1}) is given by
      \begin{equation*}
        \label{eq:31}
        F^{ij}\nabla_{k}A_{ij}= \nabla_k \psi,
      \end{equation*}
and the second covariant derivative 
      \begin{equation}
        \label{eq:32}
        F^{ij}\nabla_{k}\nabla_{k}A_{ij} + F^{ij,ml}\nabla_kA_{ij}\nabla_kA_{ml} = \nabla_{k} \nabla_k \psi.
      \end{equation}

           By multiplication of $F^{jj}$ with  \eqref{eq:30}, using coordinates such that $h_{ij} = -\delta_{ij}$ and adding repeated indices  
           \begin{multline}
\label{eq:33}
             F^{jj}\nabla_{j}\nabla_{j}A_{kk}=
            F^{jj} \nabla_{k}\nabla_{k}A_{jj} +  A_{kk}F^{jj}A_{jj}^2  - F^{jj}A_{jj}\\
             - F^{jj}A_{jj}A_{kk}^2 + A_{kk}\sum_j F^{jj}.
           \end{multline}

           Let $H = \sum_k A_{kk}$, we will use the identities above to
           compute $F^{jj}\nabla_{j}\nabla_{j}H$ that will be used later. From \eqref{eq:33} we have
           \begin{multline*}
             F^{jj}\nabla_{j}\nabla_{j}H=
             F^{jj} \sum_{k}\nabla_{k}\nabla_{k} A_{jj} +  HF^{jj}A_{jj}^2\\
             - nF^{jj}A_{jj}
             - F^{jj}A_{jj}\sum_{k} A_{kk}^2 + H\sum_j F^{jj}.
           \end{multline*}
Since $H_{k}^{1/k}$ is homogeneous of degree $1$, it holds that $F^{jj}A_{jj} = \psi$, and then
           \begin{multline*}
             F^{jj}\nabla_{j}\nabla_{j}H =
             \sum_{k}F^{jj} \nabla_{k}\nabla_{k} A_{jj}  
              +  H\left(F^{jj}A_{jj}^2 + \sum_j F^{jj}\right) - \psi\left( n+ \sum_{j} A_{jj}^2\right).
          \end{multline*}

          Using equation (\ref{eq:32}) we can rewrite the first term of the right hand side above 
          and we get
          \begin{multline}
            \label{eq:36}
     F^{jj}\nabla_{j}\nabla_{j}H=
     -\sum_{k}F^{ij,lm}\nabla_{k}A_{ij}\nabla_{k}A_{lm} + \sum_{k}\nabla_{k} \nabla_k \psi\\
     +  H\left(F^{jj}A_{jj}^2 + \sum_j F^{jj}\right) - \psi\left( n+ \sum_{j} A_{jj}^2\right).
          \end{multline}

          Now we consider the following parameterisation of the hypersurface
           \begin{equation*}
             \label{eq:37}
             Y = \sinh(u(\xi))E_1 + \cosh(u(\xi)) \xi, \quad \xi\in\mathbb{S}^n,
           \end{equation*}
           where $E_1 = (1,0,...,0)\in\mathbb{R}^{n+1,1}$. The tangent space to $\Sigma$ is spanned by the vectors
           $Y_i = u_i\left(\cosh(u)E_1 + \sinh(u)\xi\right) + \cosh(u)\xi_i = u_i \partial_r + \partial_i$. We will write
           $Y_i = \nabla_i$ and $u_i = \partial_i u = \cosh(u)\xi_i u$.

Note that
           \begin{equation*}
             \label{eq:38}
             \cosh(u)\partial_r = 
             E_1 + \sinh(u)Y.
           \end{equation*}

           The tilt and the height functions are given respectively by
           \begin{equation}
             \label{eq:39}
             \tau = \langle \hat{n} , E_1 \rangle = \frac{ \cosh^2(u)}{\sqrt{\cosh^2(u) - |\tilde{\nabla} u|^2}} \,  ;\qquad
             \eta = \langle Y, E_1\rangle= - \sinh(u),
           \end{equation}
and 
           \begin{equation*}
             \label{eq:40}
             \exp[\Phi (u,\xi)] = \frac{A_{11}}{g_{11}}\exp[\alpha(\tau)-\beta\eta].
           \end{equation*}
           \begin{prop}
             \label{prop:1}
             
             For $\tau$ and $\eta$ defined as above, the following hold:
             \begin{enumerate}
             \item\label{p:1}{$\nabla_{ij}\eta = -\tau A_{ij} - \eta g_{ij}$.}
             \item\label{p:2}{$\nabla_j \tau = -g^{ik}A_{kj}\nabla_i\eta$.}
             \item\label{p:3}{$\nabla_j\nabla_i \tau = - g^{mn}\nabla_{n}A_{ij}\nabla_{m}\eta + \tau  A_{mj}g^{mn}A_{ni} +  A_{ij} \eta$.}
             \end{enumerate}
             \end{prop}
             \begin{proof}

               Using the Weingarten equation (\ref{eq:7}) we obtain the second identity 
               \begin{equation*}
                 \begin{split}
             \label{eq:41}
             \nabla_j \tau = \langle \nabla_j \hat{n}, E_1\rangle
             &= -\langle A^{i}_{j}Y_i, E_1\rangle\\
             &= -g^{ik}A_{kj}\langle Y_i, E_1\rangle
             = -g^{ik}A_{kj}\nabla_i\langle Y, E_1\rangle
             = -g^{ik}A_{kj}\nabla_i\eta.
                 \end{split}
           \end{equation*}
           The first of the identities follows using the Gauss formula applied twice \eqref{eq:Gausscod2}
           \begin{equation*} 
             \begin{split}
             \nabla_{i}\nabla_{j}\eta & =  \langle E_1 , \nabla_i\nabla_j Y \rangle = \langle E_1 , -A_{ij}\hat{n} - g_{ij}Y \rangle  =-\tau A_{ij} - \eta g_{ij}.
             \end{split}
           \end{equation*}
           Finally, the third identity is obtained using the previous equation for the hessian of $\eta$ as follows
           \begin{equation*}
             \begin{split}
             \label{eq:43}
             \nabla_{j}\nabla_{i}\tau &= \nabla_{j}(-g^{mn}A_{ni}\nabla_{m}\eta)\\
             & = -\nabla_{j}g^{mn}A_{ni}\nabla_{m}\eta - g^{mn}\nabla_{j}A_{ni}\nabla_{m}\eta - g^{mn}A_{ni}\nabla_{mj}\eta\\
             & =- g^{mn}\nabla_{j}A_{ni}\nabla_{m}\eta - g^{mn}A_{ni}\nabla_{mj}\eta\\
             & =- g^{mn}\nabla_{n}A_{ij}\nabla_{m}\eta - g^{mn}A_{ni}(-\tau A_{mj} - \eta g_{mj})\\
             & =- g^{mn}\nabla_{n}A_{ij}\nabla_{m}\eta + \tau  A_{mj}g^{mn}A_{ni} +  g^{mn}A_{ni} \eta g_{mj}\\
             & =- g^{mn}\nabla_{n}A_{ij}\nabla_{m}\eta + \tau  A_{mj}g^{mn}A_{ni} +  A_{ij} \eta
           \end{split}
         \end{equation*}
               \end{proof}

               \begin{proof}[Proof of Theorem 1]
                 We will estimate $|H|$ and since $H^2 = |A|^2 + 2S_2$, we will get the desired estimate by admissibility. 
Since $\psi = \psi(Y,\tau)$ we first note that
\begin{equation}
\begin{split}
\nabla_k\nabla_l \psi &  = \nabla_k(\nabla_l^x \psi + \psi_{\tau}\nabla_l\tau)\\
& = \nabla_k^x\nabla_l^x\psi + \nabla_l^x\psi_{\tau}\nabla_k\tau  + \nabla_k^x\psi_{\tau}\nabla_l\tau + \psi_{\tau\tau}\nabla_{k}\tau\nabla_{l}\tau
+ \psi_{\tau}\nabla_{k}\nabla_{l}\tau,
\end{split}
\end{equation}
and also
\begin{equation}
\nabla_k^x\nabla_l^x\psi = D_k^x D_l^x\psi - (\nabla^x_{Y_{l}}Y_k )(\psi)  = -A_{kl}D^x_{\hat{n}} \psi - g_{kl}D^{x}_{X} \psi.
\end{equation}

Then, in an orthonormal frame such that $A$ is symmetric and proceeding as in \cite{Huang:01}, we have 
\begin{equation}
\sum_{k} \nabla_{k}\nabla_{k} \psi = \sum_{k} \nabla_{k}^{x}\nabla_{k}^{x} \psi + 2 \sum_{k}\nabla_k^x\psi_{\tau}\nabla_k\tau + \psi_{\tau\tau}\nabla_{k}\tau\nabla_{l}\tau +  \psi_{\tau} \sum_{k} \nabla_{k}\nabla_{k}\tau.
\end{equation}

From the assumption that $\psi$ is convex in $\tau$ and its regularity,  and Proposition~\hyperref[p:3]{\ref*{prop:1}} it follows
               \begin{equation}
            \begin{split}
            \label{eq:47}
            \sum_{k}\nabla_k\nabla_k \psi &\geq \psi_{\tau}\sum_{k}\nabla_{k}\nabla_{k}\tau + \psi_{\tau\tau}\sum_k\left(\nabla_k \tau\right)^2 - C_1H - C_2\\
            &\geq \psi_{\tau}\left(- \nabla_kH\nabla_k\eta + \tau A_{ki}A_{ki} + H\eta\right)  -C_1 H - C_2.
          \end{split}
        \end{equation}
Note that at the maximum of $H$ we have
        $\nabla H \dot{=} 0$ and $\nabla_{j}\nabla_i H \dot{\leq} 0$, then
        it follows $0\dot{\geq} F^{jj}\nabla_j\nabla_jH$. We continue from equation~(\ref{eq:36}) and using the last inequality
        (\ref{eq:47}), the concavity of $F$, the fact that $H\geq 0$
        and $\sum_jF^{jj}\geq 0$ we get
        \begin{equation*}
          \begin{split}
             \label{eq:48}
             0 & \geq
              \sum_{k}\nabla_{k} \nabla_k \psi
             +  H\left(F^{jj}A_{jj}^2 + \sum_j F^{jj}\right) - \psi\left( n+ \sum_{j} A_{jj}^2\right)\\
             & \geq
             \psi_{\tau}\left( \sum_{k} \tau A_{kk}^2 + H\eta\right)  -C_1 H - C_2
               +  HF^{jj}A_{jj}^2 - \psi\left( n+ \sum_{j} A_{jj}^2\right)\\
             & \geq -C_2  -n\psi + \left(\psi_{\tau}\eta  - C_1\right)H +  F^{jj}A_{jj}^2 H+ \left(\psi_{\tau}\tau - \psi\right)\sum_{k}A_{kk}^2.
           \end{split}
         \end{equation*}
         Since $(\psi_{\tau}\tau - \psi) \geq 0$, we can improve the last inequality by dropping the last term. Using the Newton-Maclaurin inequalities  $H_{k+1}H_{k-1} \leq H_{k}^{2}|$
         one can show (see \cite{Urbas:cespacelike}) the following
        \begin{equation*}
          \label{eq:49}
F^{ij}A_{il}A_{lj} \geq \frac{1}{n}S_{k}^{1/k}S_{1},
        \end{equation*}
and from this it follows that
\begin{equation*}
  \label{eq:50}
0 \geq -C_2  -n\psi + \left(\psi_{\tau}\eta  - C_1\right)H +  C_3\psi H^2
\end{equation*}
which implies $H$ is bounded, hence $A$ is bounded.
\end{proof}

\section{Proof of Theorem 2}
\label{sec:proof-theorem-2}

\begin{proof}

                 Consider the function $\gamma = \varphi - u$, $\gamma > 0$ in $\Omega$. Let 
                 \begin{equation*}
                   \label{eq:54}
                   \Phi(\xi) = \ln(A_{11}) + \alpha(\tau) + \beta\ln(\gamma),
                 \end{equation*}
                 its first covariant derivative  
                 \begin{equation}
                   \label{eq:55}
                   \nabla_j\Phi = \frac{\nabla_j A_{11}}{A_{11}} + \alpha'\nabla_{j}\tau + \beta\frac{\nabla_{j}\gamma}{\gamma}.
                 \end{equation}
                 The second covariant derivative is
                 \begin{multline*}
                   \nabla_{j}\nabla_{j}\Phi = \frac{\nabla_{j}\nabla_{j}A_{11}}{A_{11}} - \left(\frac{\nabla_{j}A_{11}}{A_{11}}\right)^2
                   + \alpha''\left(\nabla_{j}\tau\right)^2  \\ + \alpha'\nabla_{j}\nabla_{j}\tau + \beta\frac{\nabla_{j}\nabla_{j}\gamma}{\gamma}
                   - \beta\left(\frac{\nabla_{j}\gamma}{\gamma}\right)^2.
                 \end{multline*}
                 Using the commutator formula (\ref{eq:30}), we can replace the first term in the right hand side of the last equation,  and we also multiply the first derivatives of the
                 equation, to get an expression for  $F^{jj}\nabla_{j}\nabla_{j}\Phi$. Here, as usual, the notation indicates a sum over the repeated $j$ index. Thus we get
                 \begin{multline*}
                   F^{jj}\nabla_{j}\nabla_{j}\Phi 
                    = \frac{1}{A_{11}} \left\{ F^{jj}\nabla_{k}\nabla_{k}A_{jj} +F^{jj} A_{kk}A_{jj}^2 + F^{jj}h_{jk}h_{jk}A_{jj}\right.\\
\qquad  - F^{jj}h_{kk}h_{jj}A_{jj} - F^{jj}A_{jj}A_{kk}^2 + F^{jj}h_{jj}h_{kk}A_{kk}\\
\qquad \qquad \qquad\qquad\left. - F^{jj}h_{jk}h_{jk}A_{kk} \right\}
\                    - F^{jj}\left(\frac{\nabla_{j}A_{11}}{A_{11}}\right)^2  + \alpha''F^{jj}\left(\nabla_{j}\tau\right)^2 \\ + \alpha'F^{jj}\nabla_{j}\nabla_{j}\tau 
                    + \beta F^{jj}\frac{\nabla_{j}\nabla_{j}\gamma}{\gamma}- \beta F^{jj}\left(\frac{\nabla_{j}\gamma}{\gamma}\right)^2.
               \end{multline*}
               Note that in coordinates such that $h_{ij} = -\delta_{ij}$, some terms in the brackets 
               cancel. Now, using  the identity $F^{jj}A_{jj} = \psi$ from the homogeneity of \eqref{eq:3}, we can write
\begin{multline*}
                   F^{jj}\nabla_{j}\nabla_{j}\Phi 
                   = \frac{1}{A_{11}} F^{jj}\nabla_{1}\nabla_{1}A_{jj} +F^{jj}A_{jj}^2 - \left( A_{11} + \frac{1}{A_{11}}\right)\psi \\
                   \qquad\qquad\qquad+ \sum_j F^{jj}  
                    - F^{jj}\left(\frac{\nabla_{j}A_{11}}{A_{11}}\right)^2  + \alpha''F^{jj}\left(\nabla_{j}\tau\right)^2 \\ + \alpha'F^{jj}\nabla_{j}\nabla_{j}\tau 
                    + \beta F^{jj}\frac{\nabla_{j}\nabla_{j}\gamma}{\gamma}- \beta F^{jj}\left(\frac{\nabla_{j}\gamma}{\gamma}\right)^2.
               \end{multline*}

Using  equation (\ref{eq:32}) in the last equation we get
\begin{multline}
  \label{eq:59}
  F^{jj}\nabla_{j}\nabla_{j}\Phi 
                    = -\frac{1}{A_{11}} F^{ij,kl}\nabla_{1}A_{ij}\nabla_{1}A_{kl} + \frac{\nabla_{1}\nabla_{1}\psi}{A_{11}} \\ - \left( A_{11} + \frac{1}{A_{11}}\right)\psi +F^{jj}A_{jj}^2+ \sum_j F^{jj}  \\
\qquad\qquad\qquad\qquad                    - F^{jj}\left(\frac{\nabla_{j}A_{11}}{A_{11}}\right)^2  + \alpha''F^{jj}\left(\nabla_{j}\tau\right)^2 + \alpha'F^{jj}\nabla_{j}\nabla_{j}\tau \\
                    + \beta F^{jj}\frac{\nabla_{j}\nabla_{j}\gamma}{\gamma}- \beta F^{jj}\left(\frac{\nabla_{j}\gamma}{\gamma}\right)^2.
               \end{multline}

                              Then as in \cite{Huang:01}, by Proposition~\hyperref[p:3]{\ref*{prop:1}.(3)} and  using  $\psi(Y,\tau)$ we have
               \begin{multline*}
                 \label{eq:60}
                 \nabla_{1}\nabla_{1}\psi \geq \psi_{\tau}\nabla_{1}\nabla_{1}\tau - C_{1} A_{11} - C_{2}\\
                =  \psi_{\tau}\left(-\sum_{r}\nabla_{r}A_{11}\nabla_{r}\eta + A_{11}^2\tau + A_{11}\delta_{11}\right) - C_{1} A_{11} - C_{2}.
              \end{multline*}
              Then we have the following inequality:
              \begin{equation}
                \label{eq:61}
                \frac{\nabla_{1}\nabla_{1}\psi}{A_{11}} \geq  -\frac{\psi_{\tau}}{A_{11}}\sum_{r}\nabla_{r}A_{11}\nabla_{r}\eta + \psi_{\tau}A_{11}\tau + \psi_{\tau}\delta_{11} - C_{1} - \frac{C_{2}}{A_{11}}.
              \end{equation}

On the other hand,  using the assumption that $\varphi$ is affine then
              \begin{equation}
                \label{eq:62}
                F^{jj}\nabla_{j}\nabla_{j}\gamma \geq -C.
              \end{equation}
Also we are assuming control over $|\nabla_{j}\gamma|\leq C$, and then
                            \begin{equation}
\label{eq:63}
F^{jj}\nabla_{j}\gamma \nabla_{j}\gamma \leq C\sum_{j}F^{jj},
\end{equation}
which will be used at the end.

If we now continue using inequalities (\ref{eq:62}) and (\ref{eq:61}) in (\ref{eq:59}) we obtain

              \begin{multline*} 
\label{eq:64}
                   F^{jj}\nabla_{j}\nabla_{j}\Phi 
                   \geq -\frac{1}{A_{11}} F^{ij,kl}\nabla_{1}A_{ij}\nabla_{1}A_{kl}
                   -\frac{\psi_{\tau}}{A_{11}}\sum_{r}\nabla_{r}A_{11}\nabla_{r}\eta\\ + \psi_{\tau}A_{11}\tau + \psi_{\tau}\delta_{11} - C_{1} - \frac{C_{2}}{A_{11}}
                   +F^{jj}A_{jj}^2 \\
                   \qquad\qquad\qquad- \left( A_{11}
                   + \frac{1}{A_{11}}\right)\psi + \sum_j F^{jj}  
                 - F^{jj}\left(\frac{\nabla_{j}A_{11}}{A_{11}}\right)^2\\
                   + \alpha''F^{jj}\left(\nabla_{j}\tau\right)^2
                   + \alpha'F^{jj}\nabla_{j}\nabla_{j}\tau 
                   - \beta \frac{C}{\gamma}- \beta F^{jj}\left(\frac{\nabla_{j}\gamma}{\gamma}\right)^2.
                 \end{multline*}
                 Using again Proposition~\hyperref[p:3]{\ref*{prop:1}-(3)}, we replace the term
                  $\alpha'F^{jj}\nabla_j\nabla_{j}\tau$ to get
\begin{multline*}
                   F^{jj}\nabla_{j}\nabla_{j}\Phi 
                   \geq -\frac{1}{A_{11}} F^{ij,kl}\nabla_{1}A_{ij}\nabla_{1}A_{kl}
                   -\frac{\psi_{\tau}}{A_{11}}\sum_{r}\nabla_{r}A_{11}\nabla_{r}\eta\\
                  \qquad + \psi_{\tau}A_{11}\tau + \left(\psi_{\tau} + \alpha'\psi\right)\delta_{11} - C_{1} - \frac{C_{2}}{A_{11}} + \sum_j F^{jj} \\
\qquad\qquad\qquad                 +(1 + \alpha'\tau)F^{jj}A_{jj}^2 - \left( A_{11}
                   + \frac{1}{A_{11}}\right)\psi                   - F^{jj}\left(\frac{\nabla_{j}A_{11}}{A_{11}}\right)^2 \\
                   + \alpha''F^{jj}\left(\nabla_{j}\tau\right)^2
                   - \alpha'\sum_r\nabla_{r}\psi\nabla_{r}\eta 
                   - \beta \frac{C}{\gamma}- \beta F^{jj}\left(\frac{\nabla_{j}\gamma}{\gamma}\right)^2.
                 \end{multline*}

        Now, at the maximum, we also have
               \begin{equation*}
                 \label{eq:66}
                 -\psi_{\tau}\sum_{r}\frac{\nabla_{r}A_{11}}{A_{11}}\nabla_{r}\eta = \psi_{\tau}\sum_{r}\left( \alpha'\nabla_{r}\tau + \beta\frac{\nabla_{r}\gamma}{\gamma}\right)\nabla_{r}\eta,
               \end{equation*}
               and since $\nabla_r\psi = \psi_r + \psi_{\tau}\nabla_r\tau$, we have that
               \begin{equation*}
                 \label{eq:67}
                 -\psi_{\tau}\sum_{r}\frac{\nabla_{r}A_{11}}{A_{11}}\nabla_{r}\eta
                 - \alpha'\sum_{r}\nabla_{r}\psi\nabla_{r}\eta =
                 \sum_{r}\left(\psi_{\tau} \beta\frac{\nabla_{r}\gamma}{\gamma}  - \alpha'\psi_r\right) \nabla_{r}\eta\geq -\frac{C\beta}{\gamma} - C, 
               \end{equation*}
               then,
                                  \begin{multline}
                   \label{eq:68}
                   F^{jj}\nabla_{j}\nabla_{j}\Phi 
                   \geq -\frac{1}{A_{11}} F^{ij,kl}\nabla_{1}A_{ij}\nabla_{1}A_{kl}- C_{1} - \frac{C_2}{A_{11}}-2 \beta \frac{C}{\gamma}-C\\
\,\,                   + \left(\psi_{\tau}\tau -\psi\right)A_{11} + \left(\psi_{\tau} + \alpha'\psi\right)\delta_{11} \\
\qquad  +(1 + \alpha'\tau)F^{jj}A_{jj}^2 - \frac{\psi}{A_{11}} + \sum_j F^{jj}  \\
       - F^{jj}\left(\frac{\nabla_{j}A_{11}}{A_{11}}\right)^2
                   + \alpha''F^{jj}\left(\nabla_{j}\tau\right)^2
                   - \beta F^{jj}\left(\frac{\nabla_{j}\gamma}{\gamma}\right)^2.
                 \end{multline}

                 Case 1: There is a constant $\mu >0$ such that  
                 \begin{equation*}
                 \label{eq:69}
                 A_{nn} \leq -\mu A_{11}.
               \end{equation*}

                 In this case we will use the concavity of $F$ and drop the term with the second
                 derivatives $F^{ij,kl}$ in the  inequality~\eqref{eq:68}. Note that the last equation implies that
               \begin{equation}
                 \label{eq:70}
                 F^{jj}A_{jj}^2 \geq \frac{\mu^2}{n}A_{11}^2 \sum_{j}F^{jj},
               \end{equation}
and also
               \begin{equation*}
                 \label{eq:71}
                 F^{nn}\geq \frac{1}{n}\sum_{j}F^{jj}.
               \end{equation*}
Note as well that                                             
              \begin{equation*}
                \label{eq:72}
                F^{jj}\left(\nabla_{j}\tau\right)^2 = F^{jj}A_{jj}^2(\nabla_{j}\eta)^2 \leq C F^{jj}A_{jj}^2.
              \end{equation*}
              At the maximum of $\Phi$ we have $\nabla_{j}\Phi = 0$ and from (\ref{eq:55}) we have
               \begin{equation}
                 \label{eq:73}
              \left( \frac{\nabla_{j}A_{11}}{A_{11}} \right)^2 = \left(\alpha'\nabla_{j}\tau + \beta\frac{\nabla_{j}\gamma}{\gamma}\right)^2,
               \end{equation}
               and moreover, for all $\epsilon >0$ we have
               \begin{equation}
                 \label{eq:74}
                 \left(\alpha'\nabla_{j}\tau + \beta\frac{\nabla_{j}\gamma}{\gamma}\right)^2 <
                 (1+\epsilon)(\alpha')^2(\nabla_{j}\tau)^2 +  (1+\epsilon^{-1})\beta^2\left(\frac{\nabla_{j}\gamma}{\gamma}\right)^2.
               \end{equation}
              Note that we will find below  an $\alpha$ such that $\left(\alpha'' - (1+\epsilon)(\alpha')^2\right) < 0$, so
              \begin{equation}
                \label{eq:75}
                \left(\alpha'' - (1+\epsilon)(\alpha')^2\right) F^{jj}\left(\nabla_{j}\tau\right)^2 \geq
                  C_1\left(\alpha'' - (1+\epsilon)(\alpha')^2\right)  F^{jj}A_{jj}^2,
              \end{equation}
then from \eqref{eq:68},
              \begin{multline}
                \label{eq:76}
                   F^{jj}\nabla_{j}\nabla_{j}\Phi                  
                   \geq  -2\beta\frac{C}{\gamma} - C  - C_{1} - \frac{C_{2}}{A_{11}} 
                   + (\psi_{\tau}\tau-\psi)A_{11} \\+ (\psi_{\tau}+\alpha'\psi)\delta_{11} 
                   +\left\{(1+\alpha'\tau)+  C_1\left(\alpha'' - (1+\epsilon)(\alpha')^2\right)\right\}F^{jj}A_{jj}^2 \\
                   - \frac{\psi}{A_{11}} + \left\{1 - \left(\beta +  (1+\epsilon^{-1})\beta^2\right)\frac{1}{\gamma^2}\right\}\sum_{j}F^{jj}.
              \end{multline}

              Now, in order to control the coefficients of $F^{jj}A_{jj}^2$, we solve the
              following ordinary equation
              \begin{equation*}
                \label{eq:77}
                \alpha'' - (\alpha')^2 = 0,
              \end{equation*}
              and we find solutions of the form
              \begin{equation*}
                \label{eq:78}
                \alpha = -\ln(\tau+a),
              \end{equation*}
              where $a>0$ to be specified. Moreover, the first and second derivatives are
            \begin{equation*}
\label{eq:79}
             \alpha' = -\frac{1}{\tau+a}, \quad \alpha'' = \frac{1}{(\tau+a)^2},
            \end{equation*}
            and then it is clear that
            \begin{equation*}
\label{eq:80}
              \alpha'' - (1 + \epsilon)(\alpha')^2  = - \frac{\epsilon}{(\tau+a)^2}\leq 0,
            \end{equation*}
           from which we can also see that for $\epsilon=a^2/2C_1$ we have  
            \begin{multline*}
              (\alpha'\tau+1) + C_1(\alpha'' - (1+\epsilon)(\alpha')^2) = \frac{a}{\tau+a} - \frac{C_1\epsilon}{(\tau+a)^2}\\
=\frac{a(\tau+a)}{(\tau+a)^2} - \frac{C_1\epsilon}{(\tau+a)^2}> \frac{a^2}{2(\tau+a)^2}
              \geq C_3 >0 ,
            \end{multline*}
then form (\ref{eq:76}) we get
\begin{multline*}
  \label{eq:82}
0 \geq  -2\beta\frac{C}{\gamma} - C - C_{1} - \frac{C_{2}}{A_{11}}\\   
                   + (\psi_{\tau}\tau-\psi)A_{11} + (\psi_{\tau}+\alpha'\psi)\delta_{11} \\
                   +C_3F^{jj}A_{jj}^2 \\
                   - \frac{\psi}{A_{11}} + \left\{1 - \left(\beta +  (1+\epsilon^{-1})\beta^2\right)\frac{1}{\gamma^2}\right\}\sum_{j}F^{jj}.
               \end{multline*}

               Note $A_{11} \geq \cdots \geq A_{nn}$ and this implies that
               \begin{equation*}
                 \label{eq:83}
                 \sum_{j}F^{jj} = \frac{1}{\psi^{k-1}} H_{k-1},
               \end{equation*}
               from this it follows that
               \begin{equation*}
                 \label{eq:84}
                 \sum_{j}F^{jj} \geq C_4 > 0.
               \end{equation*}
               Using the growth assumption $\psi_{\tau}\tau-\psi >0$, the inequality \eqref{eq:70}, and choosing $\beta >0$ such that
               $\{1 - \left(\beta +  (1+\epsilon^{-1})\beta^2\right)\frac{1}{\gamma^2}\} >0$, we obtain
               \begin{equation*}
                 \label{eq:85}
                  0   \geq   -2\beta\frac{C}{\gamma} - C - C_{1} - \frac{C_{2}}{A_{11}}- \frac{\psi}{A_{11}}  +\frac{\mu^2}{n}C_3 A_{11}^2 .  
               \end{equation*}

Now we make use of the assumption $\lambda_1 \geq 1$ so that
               \begin{equation*}
                 \label{eq:86}
                 \frac{C(\beta)}{\mu} \geq \gamma A_{11}.
               \end{equation*}

               Case 2:  Looking back at inequality~\eqref{eq:68}, the assumption for this case is
               the existence of $\mu>0 $ such that 
               \begin{equation*}
                 \label{eq:87}
                 A_{nn} \geq -\mu A_{11},
               \end{equation*}
               and in this case we will make use of the term with $F^{ij,kl}$. Note also that $A_{jj}\geq -\mu A_{11}$, for all
               $j=1,2,\ldots,n$ since $A_{11}\geq A_{22}\geq \cdots \geq A_{nn}$.  

   Consider the following partition of the indices $\{ 1,2,\ldots,n\}$,
   \begin{equation*}
     \label{eq:88}
I = \{ j \,|\, F^{jj} \leq 4\,F^{11}\},\quad \mbox{and} \quad J = \{j \,|\, F^{jj}>4\,F^{11}\}.
   \end{equation*}

   Now, for $j\in I$, at the maximum, equation \eqref{eq:73} and inequality \eqref{eq:74} hold for any $\epsilon>0$, namely
   \begin{equation*}
     \label{eq:89}
                      \left(\alpha'\nabla_{j}\tau + \beta\frac{\nabla_{j}\gamma}{\gamma}\right)^2 <
                 (1+\epsilon)(\alpha')^2(\nabla_{j}\tau)^2 +  (1+\epsilon^{-1})\beta^2\left(\frac{\nabla_{j}\gamma}{\gamma}\right)^2, \quad j\in I.
   \end{equation*}

   For $j\in J$, at the maximum, since $\nabla_j\Phi =0$ in equation \eqref{eq:55},  we have for any $\epsilon >0$ that
   \begin{equation*}
     \label{eq:90}
     \beta^{-1} \left( \alpha'\nabla_{j} \tau + \frac{\nabla_{j} A_{11}}{A_{11}} \right)^2 \leq
     \frac{1+\epsilon}{\beta}(\alpha')^2(\nabla_{j}\tau)^2 + \frac{1+\epsilon^{-1}}{\beta}\left(\frac{\nabla_{j}A_{11}}{A_{11}}\right)^2.
   \end{equation*}

   From these two inequalities we can get
   \begin{equation*}   
     \begin{split}
     \label{eq:91}
     \beta F^{jj}\left(\frac{\nabla_{j}\gamma}{\gamma}\right)^2 + F^{jj}\left(\frac{\nabla_{j}A_{11}}{A_{11}}\right)^2 &\leq
      \frac{1+\epsilon}{\beta}(\alpha')^2\sum_{j\in J}F^{jj}(\nabla_{j}\tau)^2\\ &\quad + \frac{1+\epsilon^{-1}}{\beta}\sum_{j\in J}F^{jj}\left(\frac{\nabla_{j}A_{11}}{A_{11}}\right)^2\\
&\quad + \beta\sum_{j\in I}F^{jj}\left(\frac{\nabla_{j}\gamma}{\gamma}\right)^2      + \sum_{j\in J} F^{jj}\left( \frac{\nabla_{j} A_{11}}{A_{11}}\right)^2\\
&\quad + (1+\epsilon)(\alpha')^2\sum_{j\in I}F^{jj}(\nabla_{j}\tau)^2\\
&\quad+  (1+\epsilon^{-1})\beta^2\sum_{j\in I}F^{jj}\left(\frac{\nabla_{j}\gamma}{\gamma}\right)^2\\
& \leq
4n\{ \beta + (1 + \epsilon^{-1})\beta^2\} F^{11}\left(\frac{\nabla_{j}\gamma}{\gamma}\right)^2\\
&\quad+ (1+\epsilon)(1+\beta^{-1})(\alpha')^2F^{jj}(\nabla_{j}\tau)^2\\
& \quad    + \{ 1 + (1+\epsilon^{-1})\beta^{-1}\}\sum_{j\in J}F^{jj}\left(\frac{\nabla_{j}A_{11}}{A_{11}}\right)^2.
   \end{split}
 \end{equation*}


   Using the last two estimates in \eqref{eq:68} at the maximum we obtain
   \begin{multline*}
                   0\geq -\frac{1}{A_{11}} F^{ij,kl}\nabla_{1}A_{ij}\nabla_{1}A_{kl}- C_{1} - \frac{C_2}{A_{11}}-2 \beta \frac{C}{\gamma}-C\\
                   + \left(\psi_{\tau}\tau -\psi\right)A_{11} + \left(\psi_{\tau} + \alpha'\psi\right)\delta_{11} \\
\qquad                   +(1 + \alpha'\tau)F^{jj}A_{jj}^2 - \frac{\psi}{A_{11}} + \sum_j F^{jj}  \\
                   - F^{jj}\left(\frac{\nabla_{j}A_{11}}{A_{11}}\right)^2
                   + \alpha''F^{jj}\left(\nabla_{j}\tau\right)^2
                   - \beta F^{jj}\left(\frac{\nabla_{j}\gamma}{\gamma}\right)^2.
                 \end{multline*}
Solving $\alpha'' - (\alpha')^2 =0$ as in Case 1, we obtain \eqref{eq:75}, then 
                 \begin{multline*}
                   0\geq -\frac{1}{A_{11}} F^{ij,kl}\nabla_{1}A_{ij}\nabla_{1}A_{kl}- C_{1} - \frac{C_2}{A_{11}}-2 \beta \frac{C}{\gamma}-C\\
                   + \left(\psi_{\tau}\tau -\psi\right)A_{11} + \left(\psi_{\tau} + \alpha'\psi\right)\delta_{11} - \frac{\psi}{A_{11}}  \\
\qquad- 4n\{ \beta+(1+\epsilon^{-1})\beta^2\}F^{11}\left(\frac{\nabla_{j}\gamma}{\gamma}\right)^2 + \sum_j F^{jj}\\
 \qquad\qquad\qquad +\{(1 + \alpha'\tau) + C_1\left( \alpha'' - (1+\epsilon)(1+\beta^{-1})(\alpha')^2\right)\}F^{jj}A_{jj}^2  \\
     - \{ 1 + (1+\epsilon^{-1})\beta^{-1}\}\sum_{j\in J}F^{jj}\left(\frac{\nabla_{j}A_{11}}{A_{11}}\right)^2,
   \end{multline*}
   and moreover, for $\epsilon = \epsilon(a)$, there is a $C_0>0$ such that the last term
   is improved by
   \begin{multline}
     \label{eq:94}
                   0\geq -\frac{1}{A_{11}} F^{ij,kl}\nabla_{1}A_{ij}\nabla_{1}A_{kl}- C_{1} - \frac{C_2}{A_{11}}-2 \beta \frac{C}{\gamma}-C\\
                   + \left(\psi_{\tau}\tau -\psi\right)A_{11} + \left(\psi_{\tau} + \alpha'\psi\right)\delta_{11} - \frac{\psi}{A_{11}}  \\
\qquad- 4n\{ \beta+(1+\epsilon^{-1})\beta^2\}F^{11}\left(\frac{\nabla_{j}\gamma}{\gamma}\right)^2 + \sum_j F^{jj}\\
 \qquad\qquad\qquad +\{(1 + \alpha'\tau) + C_1\left( \alpha'' - (1+\epsilon)(1+\beta^{-1})(\alpha')^2\right)\}F^{jj}A_{jj}^2  \\
     - \{ 1 + C_0\beta^{-1}\}\sum_{j\in J}F^{jj}\left(\frac{\nabla_{j}A_{11}}{A_{11}}\right)^2.
   \end{multline}
   

   It is also known (see for instance Lemma 2.20 and Lemma 2.21 in \cite{Andrews1994}) that for any symmetric matrix $\eta_{ij}$ we have
   \begin{equation*}
\label{eq:95}
     F^{ij,kl}\eta_{ij}\eta_{kl} = \frac{\partial^2 f}{\partial \lambda_i\partial\lambda_j}\eta_{ii}\eta_{jj} +
     \sum_{i\neq j}\frac{f_i - f_j}{\lambda_i-\lambda_j}\eta_{ij}^2,
   \end{equation*}
   and whenever $F$ is concave, then the second term of the right hand side of the equation is non-positive and it
   should be read as a limit when $\lambda_i = \lambda_j$.
   Then, using this Lemma, the Codazzi equation \eqref{eq:13} and since $1\notin J$ we have the following inequality
   \begin{multline*}
     -\frac{1}{\lambda_1}F^{ij,kl}\nabla_{1}A_{ij}\nabla_1A_{kl} \geq
     -\frac{2}{\lambda_1}\sum_{j\in J}\frac{f_{1}-f_{j}}{\lambda_1 - \lambda_j}|\nabla_1A_{1j}|^2\\ =
     -\frac{2}{\lambda_1}\sum_{j\in J}\frac{f_{1}-f_{j}}{\lambda_1 - \lambda_j}|\nabla_jA_{11}|^2.
   \end{multline*}
Then following from \eqref{eq:94} we get
   \begin{multline}
     \label{eq:97}
                        0\geq - C_{1} - \frac{C_2}{A_{11}}-2 \beta \frac{C}{\gamma}-C
                   + \left(\psi_{\tau}\tau -\psi\right)A_{11} + \left(\psi_{\tau} + \alpha'\psi\right)\delta_{11} \\
\qquad\quad+ C_{3}F^{jj}A_{jj}^2 - \frac{\psi}{A_{11}} + \sum_j F^{jj}  
          - 4n\{ \beta + (1 + \epsilon^{-1})\beta^2\} F^{11}\left(\frac{\nabla_{j}\gamma}{\gamma}\right)^2 \\
          - \left( 1 + C_0\beta^{-1}\right)\sum_{j\in J}F^{jj}\left(\frac{\nabla_{j}A_{11}}{A_{11}}\right)^2
               -\frac{2}{\lambda_1}\sum_{j\in J}\frac{f_{1}-f_{j}}{\lambda_1 - \lambda_j}|\nabla_jA_{11}|^2.
   \end{multline}

   Put $\delta = C_0 \beta^{-1}$, and recall that since $j \in J$ we have $f_j > 4f_1$. If $\lambda_j > 0$ then the equation
   \begin{equation}
     \label{eq:98}
     (1-\delta) f_j \lambda_1 \geq 2f_1\lambda_1 - (1+\delta) f_j\lambda_j, \quad \mbox{for }\quad j \in J,
   \end{equation}
   holds with $\delta = \frac{1}{4}$. If $\lambda_j \leq 0$, then since $\lambda_n \geq -\mu \lambda_1$ and
   thus $\lambda_j \geq -\mu \lambda_1$ for all $j = 1,2,\ldots,n$, then we have $|\lambda_j|\leq\mu\lambda_1$.
   This implies that \eqref{eq:98} is also satisfied if $\delta = 1/4 $ and $\mu = 1/5$. Recall that this
   choices implies a value for $\beta$ which depends on $\sup_{\Omega}|\tilde{\nabla} u|$.

   Equation \eqref{eq:98} implies the inequality
   \begin{equation*}
     \label{eq:99}
     -\frac{2}{\lambda_1}\frac{f_1-f_j}{\lambda_1-\lambda_j} \geq (1+C_0\beta^{-1})\frac{f_j}{\lambda_1^2}, \quad j\in J,
   \end{equation*}
   for $\beta$ sufficiently small, and then we can drop the last two terms in \eqref{eq:97}

   \begin{multline*}
                             0\geq - C_{1} - \frac{C_2}{A_{11}}-2 \beta \frac{C}{\gamma}-C\\
                   + \left(\psi_{\tau}\tau -\psi\right)A_{11} + \left(\psi_{\tau} + \alpha'\psi\right)\delta_{11} 
+ C_{3}F^{jj}A_{jj}^2 \\- \frac{\psi}{A_{11}} + \sum_j F^{jj}  
          - 4n\{ \beta + (1 + \epsilon^{-1})\beta^2\} F^{11}\left(\frac{\nabla_{j}\gamma}{\gamma}\right)^2
   \end{multline*}
Now, recall from \eqref{eq:63} we get
   \begin{multline*}
     \label{eq:101}
                             0\geq - C_{1} - \frac{C_2}{A_{11}}-2 \beta \frac{C}{\gamma}-C
                   + \left(\psi_{\tau}\tau -\psi\right)A_{11} + \left(\psi_{\tau} + \alpha'\psi\right)\delta_{11} \\
\qquad+ C_{3}F^{jj}A_{jj}^2 - \frac{\psi}{A_{11}} + \sum_{j}F^{jj}  
          - 4n\{ \beta + (1 + \epsilon^{-1})\beta^2\} C \frac{F^{11}}{\gamma^2}, \\
   \end{multline*}
   which gives us at the end an estimate of the type
   \begin{equation*}
     \label{eq:102}
     C_4\lambda_1  + C_3 F^{11}\lambda_1^2 \leq C\left(1 + \frac{1}{\gamma}+\frac{F^{11}}{\gamma^2}\right),
   \end{equation*}
which concludes the proof the theorem.

\end{proof}



\begin{thebibliography}{10}

\bibitem{Andrews1994}
Ben Andrews.
\newblock Contraction of convex hypersurfaces in euclidean space.
\newblock {\em Calculus of Variations and Partial Differential Equations},
  2(2):151--171, May 1994.

\bibitem{Oliker1}
J.~Lucas~M. Barbosa, Jorge H.~S. Lira, and Vladimir~I. Oliker.
\newblock A priori estimates for starshaped compact hypersurfaces with
  prescribed {$m$}th curvature function in space forms.
\newblock In {\em Nonlinear problems in mathematical physics and related
  topics, {I}}, volume~1 of {\em Int. Math. Ser. (N. Y.)}, pages 35--52.
  Kluwer/Plenum, New York, 2002.

\bibitem{bartnik1982}
Robert Bartnik and Leon Simon.
\newblock Spacelike hypersurfaces with prescribed boundary values and mean
  curvature.
\newblock {\em Comm. Math. Phys.}, 87(1):131--152, 1982.

\bibitem{CNS3}
L.~Caffarelli, L.~Nirenberg, and J.~Spruck.
\newblock The {D}irichlet problem for nonlinear second-order elliptic
  equations. {III}. {F}unctions of the eigenvalues of the {H}essian.
\newblock {\em Acta Math.}, 155(3-4):261--301, 1985.

\bibitem{CNS4}
L.~Caffarelli, L.~Nirenberg, and J.~Spruck.
\newblock Nonlinear second order elliptic equations. {IV}. {S}tarshaped compact
  {W}eingarten hypersurfaces.
\newblock In {\em Current topics in partial differential equations}, pages
  1--26. Kinokuniya, Tokyo, 1986.

\bibitem{CG:Lorentz}
C.~Gerhardt.
\newblock Hypersurfaces of prescribed curvature in {L}orentzian manifolds.
\newblock {\em Indiana Univ. Math. J.}, 49:1125--1153, 2000.

\bibitem{Gerhardt1983}
Claus Gerhardt.
\newblock H-surfaces in Lorentzian manifolds.
\newblock {\em Communications in Mathematical Physics}, 89(4):523--553, Dec
  1983.

\bibitem{Gerhardt1997}
Claus Gerhardt.
\newblock Hypersurfaces of prescribed weingarten curvature.
\newblock {\em Mathematische Zeitschrift}, 224(2):167--194, Feb 1997.

\bibitem{Huang:01}
Yong Huang.
\newblock Curvature estimates of hypersurfaces in the {M}inkowski space.
\newblock {\em Chin. Ann. Math. Ser. B}, 34(5):753--764, 2013.

\bibitem{YYLi:hyp}
Qinian Jin and YanYan Li.
\newblock Starshaped compact hypersurfaces with prescribed $k$-th mean
  curvature in hyperbolic space.
\newblock {\em Discrete \& Continuous Dynamical Systems - A}, 15:367, 2006.

\bibitem{YYLi:elli}
Yanyan Li and Vladimir~I. Oliker.
\newblock Starshaped compact hypersurfaces with prescribed {$m$}-th mean
  curvature in elliptic space.
\newblock {\em J. Partial Differential Equations}, 15(3):68--80, 2002.

\bibitem{sheng2004}
Weimin Sheng, John Urbas, and Xu-Jia Wang.
\newblock Interior curvature bounds for a class of curvature equations.
\newblock {\em Duke Math. J.}, 123(2):235--264, 06 2004.

\bibitem{Urbas:cespacelike}
John Urbas.
\newblock Interior curvature bounds for spacelike hypersurfaces of prescribed
  {$k$}-th mean curvature.
\newblock {\em Comm. Anal. Geom.}, 11(2):235--261, 2003.

\end{thebibliography}
\end{document}